\documentclass[12pt]{amsart}
\usepackage{amscd,amssymb,hyperref}
\usepackage[dvips]{graphics}
\input xy
\xyoption{all}
\usepackage{graphics}
\usepackage{epsfig}
\usepackage{picinpar}
\usepackage{wrapfig}
\usepackage{amsmath}
\usepackage{color}
\usepackage{pstricks,pst-node,pst-text,pst-3d}
\usepackage{pst-plot}
\usepackage{verbatim}

\newtheorem{thm}{Theorem}[section]
\newtheorem{lem}[thm]{Lemma}

\newtheorem{prop}[thm]{Proposition}

\def\square{\vbox{
      \hrule height 0.4pt
      \hbox{\vrule width 0.4pt height 5.5pt \kern 5.5pt \vrule width 0.4pt}
      \hrule height 0.4pt}}

\def\ch\mathrm{c h}

\newcommand{\Z}{\mathbb{Z}}

\let\la=\langle
\let\ra=\rangle

\numberwithin{equation}{section}

\begin{document}

\newcommand{\auths}[1]{\textrm{#1},}
\newcommand{\artTitle}[1]{\textsl{#1},}
\newcommand{\jTitle}[1]{\textrm{#1}}
\newcommand{\Vol}[1]{\textbf{#1}}
\newcommand{\Year}[1]{\textrm{(#1)}}
\newcommand{\Pages}[1]{\textrm{#1}}

\author{V.~G.Bardakov}
\address{Sobolev Institute of Mathematics, Novosibirsk 630090, Russia}
\email{bardakov@math.nsc.ru}

\author{R. Mikhailov}
\address{Steklov Mathematical Institute, Gubkina 8, 119991 Moscow, Russia}
\email{rmikhailov@mail.ru}

\author{V.~V.~Vershinin}
\address{D\'epartement des Sciences Math\'ematiques,
                                     Universit\'e Montpellier II,
Place Eug\`ene Bataillon,
34095 Montpellier cedex 5, France}
\email{ vershini@math.univ-montp2.fr}
\address{Sobolev Institute of Mathematics, Novosibirsk 630090,
Russia }
\email{ versh@math.nsc.ru}

\author{J. Wu}

\address{Department of Mathematics, National University of Singapore, 2 Science Drive 2
Singapore 117542} \email{matwuj@nus.edu.sg}
\urladdr{www.math.nus.edu.sg/\~{}matwujie}

\thanks{This work was started during the stay
of the four authors in Oberwolfach in the framework of the pro gramme RiP
during May 18th - June 7th, 2008. The authors would like to thank the Mathematical Institute of Oberwolfach for its hospitality.
The work was continued during the visit of the
 third author in the Institute for Mathematical Sciences of the National University of Singapore during 
December 4 - 19,  2008. He would like to express his gratitude to the
IMS
}

\thanks{The last author is partially supported by the Academic Research Fund
of the National University of Singapore R-146-000-101-112}

\begin{abstract}
In this article, we investigate various properties of the pure virtual braid group $PV_3$. From its canonical presentation, we obtain a free product decomposition of $PV_3$. As a consequence, we show that $PV_3$ is residually torsion free nilpotent, which implies that the set of finite type
invariants in the sense of Goussarov-Polyak-Viro is complete for virtual pure braids with three strands. Moreover we prove that the presentation of $PV_3$ is aspherical. Finally we determine the cohomology ring and the associated graded Lie algebra of $PV_3$.
\end{abstract}

\title{On the pure virtual braid group $PV_3$}



\maketitle

\section{Introduction}

Virtual knots were introduced by L.~Kauffman \cite{Ka} and studied by many authors.
One of the  motivations lies in the
theory of Gauss diagrams and Gauss codes of knots \cite{Ka}, \cite{PV}.
Namely, for any
knot diagram it is possible to construct its Gauss diagram and form its Gauss
code.
The problem is that not every Gauss diagram (or Gauss code) corresponds
to some knot. To escape this difficulty virtual knots were introduced.
Strictly
virtual knots are those whose Gauss diagram is not a Gauss diagram of any
classical knot.
Many notions from the classical knot
theory are generalized to virtual knots, such as fundamental group, rack, quandle, Kauffman
and Jones polynomials.
  M.~Goussarov, M.~Polyak and O.~Viro \cite{GPV} proved that
the analogues of the upper and the lower presentations of the classical
fundamental group of a knot give two different groups for virtual knot.
Theorem of M.~Goussarov, M.~Polyak and O.~Viro \cite{GPV} says that two virtually equivalent classical knots are classically equivalent. This theorem means that inclusion of usual knots in the universe of virtual knots does not spoil the classical theory. On the other hand virtual knots are appeared to be
useful: combinatorial formulas for Vassiliev invariants  were obtained
by means of virtual knot theory.

Virtual braid groups $VB_n$ were introduced in \cite{Ka} and \cite{Ve}.
Seichi Kamada \cite{Kam} proved that any virtual link
can be described as the closure of a virtual braid, which is unique up to certain basic moves. This is analogous to the Alexander and Markov theorems for classical braids and links.
So, the same way as in classical case virtual braids can be used in the study of virtual links.

In the {\it virtual braid group} two types of crossings are
allowed: 1) as usual braids, or 2) as an
intersection of lines on the plane.
This group is given by the following set of generators:
$\{\zeta_i, \sigma_i, \ \ i=1,2, ..., n-1 \}$ and relations:
$$ \begin{cases} \zeta_i^2&=1, \\
\zeta_i \zeta_j &=\zeta_j \zeta_i, \ \ \text {if} \ \ |i-j| >1,\\
\zeta_i \zeta_{i+1} \zeta_i &= \zeta_{i+1} \zeta_i \zeta_{i+1}. \end{cases} $$
\centerline { The symmetric group relations }

$$ \begin{cases} \sigma_i \sigma_j &=\sigma_j \sigma_i, \ \text {if} \  |i-j| >1,
\\ \sigma_i \sigma_{i+1} \sigma_i &= \sigma_{i+1} \sigma_i \sigma_{i+1}.
\end{cases} 
$$

\centerline {The braid group relations }

$$ \begin{cases} \sigma_i \zeta_j &=\zeta_j \sigma_i, \ \text {if} \  |i-j| >1,\\
\zeta_i \zeta_{i+1} \sigma_i &= \sigma_{i+1} \zeta_i \zeta_{i+1}.
\end{cases} $$
\vglue 0.1cm
\centerline {The mixed relations }

The generator $\sigma_i$ corresponds to the canonical generator of the braid
group $Br_n$.
The generators $\zeta_i$ correspond to the intersection of lines.

As for the classical braid groups there exists the canonical
epimorphism to the symmetric group $VB_n\to \Sigma_n$ with the
kernel called the pure virtual braid group $PV_n$. So we have a
short exact sequence
\begin{equation*}
1 \to PV_n \to VB_n \to \Sigma_n \to 1.
\end{equation*}
Define the following elements in $PV_n$
$$
\lambda_{i,i+1} = \rho_i \, \sigma_i^{-1},~~~
\lambda_{i+1,i} = \rho_i \, \lambda_{i,i+1} \, \rho_i = \sigma_i^{-1} \, \rho_i,
~~~i=1, 2, \ldots, n-1,
$$
$$
\lambda_{ij} = \rho_{j-1} \, \rho_{j-2} \ldots \rho_{i+1} \, \lambda_{i,i+1} \, \rho_{i+1}
\ldots \rho_{j-2} \, \rho_{j-1},
$$
$$
\lambda_{ji} = \rho_{j-1} \, \rho_{j-2} \ldots \rho_{i+1} \, \lambda_{i+1,i} \, \rho_{i+1}
\ldots \rho_{j-2} \, \rho_{j-1}, ~~~1 \leq i < j-1 \leq n-1.
$$
These elements belong to $VP_n$.
It is shown in \cite{B} that the group $PV_n\ (n\geq 3)$ admits a
presentation with the  generators $\lambda_{ij},\ 1\leq i\neq j\leq n,$
and the following relations: \begin{align}
& \lambda_{ij}\lambda_{kl}=\lambda_{kl}\lambda_{ij}\\
&
\lambda_{ki}\lambda_{kj}\lambda_{ij}=\lambda_{ij}\lambda_{kj}\lambda_{ki}\label{relation},
\end{align}
where distinct letters stand for distinct indices.

Like the usual pure braid groups, groups $PV_n$ admit a
semi-direct product decompositions \cite{B}: for $n\geq 2,$ the
$n$-th virtual pure braid group can be decomposed as
$$
PV_n=V_{n-1}^*\rtimes PV_{n-1},
$$
where $V_{n-1}^*$ is a free subgroup of $PV_{n-1}$.

As it happens in Mathematics the virtual braid groups appeared
in another context under the other name  ``$n$-th quasitriangular
group $\mathbf{QTr}_n$ " in the work of
 L. Bartholdi, B. Enriquez, P. Etingof and E. Rain
\cite{BEER} as groups associated
to the Yang-Baxter equations.
Really consider the equation (1.2) 
for $k=1$, $i=2$, $j=3$, we get the quantum Yang-Baxter equation
 \begin{equation*}
\lambda_{12}\lambda_{13}\lambda_{23}=\lambda_{23}\lambda_{13}\lambda_{12}.
\end{equation*}
Also the authors of introduce
quadratic Lie algebras ${\mathfrak{qtr}}_n$
and map them onto the associated graded algebras
of the Malcev Lie algebras of the groups
${\mathbf{QTr}}_n$.
Among the other results of \cite{BEER}, the complete
description of integral homology groups of $PV_n$ is given:
the homology group
$H_r(PV_n,\mathbb Z)$ is free abelian of rank equal to
$\binom{n-1}{r}\frac{n!}{(n-r)!}$
(the number of unordered
partitions of $n$ in $r$ ordered parts).

\section{Residually Nilpotence of $PV_3$}
The group $PV_3$ admits the following presentation:
\begin{multline}\label{asphpres}\langle \lambda_{12},
\lambda_{21},\lambda_{13},\lambda_{31},\lambda_{23},\lambda_{32}\
|\ \lambda_{12} \lambda_{13}
\lambda_{23}\lambda_{12}^{-1}\lambda_{13}^{-1}\lambda_{23}^{-1}
=1\\ \lambda_{21} \lambda_{23}
\lambda_{13}\lambda_{21}^{-1}\lambda_{23}^{-1}\lambda_{13}^{-1} =
1,\
 \lambda_{13} \lambda_{12} \lambda_{32}\lambda_{13}^{-1}\lambda_{12}^{-1}\lambda_{32}^{-1}=1,\\ \lambda_{31} \lambda_{32}
\lambda_{12}\lambda_{31}^{-1}\lambda_{32}^{-1}\lambda_{12}^{-1} =
1,\
 \lambda_{23} \lambda_{21} \lambda_{31} \lambda_{23}^{-1}\lambda_{21}^{-1}\lambda_{31}^{-1}=1,\\ \lambda_{32} \lambda_{31}
\lambda_{21}
\lambda_{32}^{-1}\lambda_{31}^{-1}\lambda_{21}^{-1}=1\rangle.
\end{multline}

Recall that
\begin{equation}\label{semi}
PV_3 = V_2^* \leftthreetimes V_1,
\end{equation}
where $V_1 = \langle \lambda_{12}, \lambda_{21} \rangle$ is a free
2-generated group and $V_2^*$ is the normal closure of group $V_2
= \langle \lambda_{13}, \lambda_{23},  \lambda_{31}, \lambda_{32}
\rangle$ in $PV_3$. The group $V_2$ is a free  of rank 4.

\begin{prop}\label{decomp}
There exists a group $G_3$, such that there is the following free
product decomposition: $PV_3\cong G_3\ast \Z$.
\end{prop}

\begin{proof}
We can represent generators $\lambda_{ij}$ as virtual braids (see
\cite{B}) and use the operations of doubling of string we
construct new generators of $PV_3$ from generators of $V_1$.

We have
$$
a_1 =
\lambda_{13} \lambda_{23},~~~b_1 =
\lambda_{13} \lambda_{12},
$$
$$
b_2 =
\lambda_{21} \lambda_{31},~~~a_2 =
\lambda_{32} \lambda_{31},
$$
and it is easy to check that
$$
PV_3 = \langle a_1, a_2, b_1, b_2, \lambda_{13}, \lambda_{31}
\rangle.
$$

We can write the  old generators as words in new generators:
$$
\lambda_{12} = \lambda_{13}^{-1} b_1,~~~\lambda_{21} = b_2
\lambda_{31}^{-1},~~~ \lambda_{23} = \lambda_{13}^{-1}
a_1,~~~\lambda_{32} = a_2 \lambda_{31}^{-1}.
$$
The relations from the presentation (\ref{relation}) of $PV_3$  in
new generators have the following form
$$
a_1 b_1 = b_1 a_1,~~~b_1^{a_2} = b_1^{\lambda_{13} \lambda_{31}},
$$
$$
b_2^{a_1^{-1}} = b_2^{(\lambda_{13}
\lambda_{31})^{-1}},~~~a_1^{b_2} = a_1^{\lambda_{13}
\lambda_{31}},
$$
$$
a_2^{b_1^{-1}} = a_2^{(\lambda_{13} \lambda_{31})^{-1}},~~~a_2 b_2
= b_2 a_2,
$$
where $y^x \, = \,  x^{-1} \, y \, x.$

If we define $c_1 = \lambda_{13} \lambda_{31}$, $c_2 =
\lambda_{13},$ then
\begin{multline*}
PV_3 = \langle a_1, a_2, b_1, b_2, c_1, c_2\ |\ [a_1, b_1] = [a_2, b_2] = 1, \\
b_1^{c_1} = b_1^{a_2},~~~a_1^{c_1} = a_1^{b_2},~~~b_2^{c_1} =
b_2^{a_1 b_2},~~~a_2^{c_1} = a_2^{b_1 a_2} \rangle,
\end{multline*}
where $[x, y]$ as usual denotes the commutator of the elements $x$ and $y$:
$[x, y]= x^{-1}y^{-1}xy $.
Hence $PV_3 = G_3 * \langle c_2 \rangle$, where the group
$G_3$
can be presented as having generators $ a_1, a_2, b_1, b_2, c_1$
and the following relations
\begin{equation}
[a_1, b_1] = [a_2, b_2] = 1, 
\end{equation}
\begin{equation}
b_1^{c_1} = b_1^{a_2},~~~a_1^{c_1} = a_1^{b_2},~~~b_2^{c_1} =
b_2^{a_1 b_2},~~~a_2^{c_1} = a_2^{b_1 a_2} \label{eq:relcong}
\end{equation}
 and $G_3 = Q_3 \leftthreetimes \langle c_1 \rangle,$ where $Q_3$ is a subgroup
of $G_3$ with the set of generators $a_1, b_1, a_2, b_2$ and with
infinite set of relations
\begin{equation}
[a_i, b_i]^{c_1^k} = 1,~~~i = 1, 2,~~~k \in \mathbb{Z},
\end{equation}
rewritten in generators $a_1, b_1, a_2, b_2$ using relations
(\ref{eq:relcong}).
\end{proof}

The main result of this section is the following:

\begin{thm}\label{resnilp}
The  pure virtual braid group $PV_3$ is residually torsion free
nilpotent.
\end{thm}

Before proof of this theorem let us make a short deviation and
observe that the semi-direct product decomposition (\ref{semi}) of
$PV_3$ does not imply the residual nilpotence of $PV_3$
immediately.

\subsection{Remarks of Non-residually Nilpotent Groups}

Let $G_1$ and $G_2$ be residually nilpotent groups and let $G =
G_1 \leftthreetimes G_2$ be a semi-direct product of these groups.
The following question naturally arises: when $G$  is residually
nilpotent? As was proved in \cite{FR1} if $G_2$ acts trivially on
the abelianization of $G_1$ then the answer is positive.

We consider the case when $G_1$ and $G_2$ are free. Let $G_2 =
\mathbb{Z}$ be an infinite cyclic group. Then the  group $G$ is
called {\it the mapping torus} and have received a great deal of
attention in recent years.

Let $G_{\varphi} = F_n \leftthreetimes_{\varphi} \mathbb{Z}$,
where $ {\varphi}$ is a homomorphism ${\varphi}: \mathbb{Z}\to
\mathrm{Aut}(F_n)$.
In
the case $n = 1$ it is easy to prove that $G_{\varphi}$ is
residually nilpotent. In this case $\mathrm{Aut}(F_1)$ contains
only two automorphisms: trivial and inversion: $\varphi (a) =
a^{-1},$ where $F_1 = \langle a \rangle$. In the first case
$G_{\varphi} = F_1 \times \mathbb{Z} = \mathbb{Z}^2$  is abelian
and hence is residually nilpotent. In the second case it easy
calculate that
$$
G_{\varphi}' = \gamma_2 G_{\varphi} = \{ a^{2k} ~|~ k \in
\mathbb{Z} \}
$$
and if $m > 2$ then
$$
\gamma_m G_{\varphi} = \{ a^{2^{m-1} k} ~|~ k \in \mathbb{Z} \}.
$$
Hence
$$
\bigcap_{i=1}^{\infty} \gamma_i G_{\varphi} = 1,
$$
and $G_{\varphi}$ is residually nilpotent.

If $G_{\varphi} = F_2 \leftthreetimes_{\varphi} \mathbb{Z}$ then
the following example shows that there exists an automorphism $\varphi$
for which our  group is not residually nilpotent.

\noindent{\bf Example} (see \cite{MP}{, Example 1.32}). Let $G =
F_{2} \leftthreetimes \langle t\rangle$, where $F_{2} = \langle a,
b \rangle$ and $t$ is the following automorphism of $F_{2}$:
$$
t : \left\{
\begin{array}{ccl}
a & \longrightarrow  & a^2 \, b, \\
b & \longrightarrow  & a \, b. \\
\end{array}
\right.
$$
From the relation
$$
t \, b \, t^{-1} = a \, b
$$
we have
$$
[t^{-1}, b^{-1}] = a.
$$
From the relation
$$
t \, a \, t^{-1} = a^2 \, b
$$
we have
$$
b = a^{-1} \, [a, t^{-1}] = [b^{-1}, t^{-1}] \, [a, t^{-1}].
$$
Hence in this example
$$
\bigcap_{i=1}^{\infty} \gamma_i G = F_2.
$$
and the group $G$ is not residually nilpotent.

We can generalize this example. Let us consider now the following
group
$$
G_{\varphi} = F_n \leftthreetimes \langle t \rangle,~~ n \geq 2,
$$
where conjugation by $t$ induces the automorphism $\varphi \in
\mathrm{Aut}(F_n).$ Let $A = [\varphi]$ be the abelianization of
$\varphi$ i.~e. $A \in \mathrm{GL}_n(\mathbb{Z}) \simeq
\mathrm{Aut(F_n/F_n')}$ is induced by $\varphi.$ Let $E$ be the
identity matrix from $\mathrm{GL}_n(\mathbb{Z})$.

\begin{prop}
If the matrix $A-E \in \mathrm{GL}_n(\mathbb{Z}),$ then
$\bigcap_{i=1}^{\infty} \gamma_i (G_{\varphi}) = F_n$ and
$G_{\varphi}$ is not residually nilpotent.
\end{prop}
\begin{proof} Let
$$
\varphi : \left\{
\begin{array}{l}
x_1 \longrightarrow x_1^{\alpha_{11}} \, x_2^{\alpha_{12}} \, \ldots \, x_n^{\alpha_{1n}} \, c_1, \\
x_2 \longrightarrow x_1^{\alpha_{21}} \, x_2^{\alpha_{22}} \, \ldots \, x_n^{\alpha_{2n}} \, c_2, \\
\ldots\\
x_n \longrightarrow x_1^{\alpha_{n1}} \, x_2^{\alpha_{n2}} \, \ldots \, x_n^{\alpha_{nn}} \, c_n, \\
\end{array}
\right.
$$
where $\alpha_{ij} \in \mathbb{Z}$, $c_i \in F_n'.$ Then
$[\varphi] = (\alpha_{ij})_{i,j =1}^n = A.$ Since $\varphi \in
\mathrm{Aut}(F_n)$ then $A \in \mathrm{GL}_n(\mathbb{Z}).$
Consider the commutator
$$
[x_i, \varphi] = x_i^{-1} (x_i)^{\varphi} = x_1^{\alpha_{i1}} \,
\ldots \, x_i^{\alpha_{ii} -1} \, \ldots \, x_n^{\alpha_{in}} \,
c'_i,~~ c'_i \in F_n',~~~i = 1, 2, \ldots, n.
$$
Hence we have the following system
$$
\left\{
\begin{array}{l}
x_1^{\alpha_{11}-1} \, x_2^{\alpha_{12}} \, \ldots \, x_n^{\alpha_{1n}} = d_1, \\
 x_1^{\alpha_{21}} \, x_2^{\alpha_{22}-1} \, \ldots \, x_n^{\alpha_{2n}} = d_2, \\
\ldots \\
 x_1^{\alpha_{n1}} \, x_2^{\alpha_{n2}} \, \ldots \, x_n^{\alpha_{nn}-1} = d_n, \\
\end{array}
\right.
$$
where  $d_i = [x_i, \varphi] (c_i')^{-1} \in G_{\varphi}'$. Since
$A - E \in \mathrm{GL}_n(\mathbb{Z})$ then from this system we
have that all $x_i$ lie in the commutator subgroup $G_{\varphi}'$
but $G_{\varphi}' \leq F_n$ Hence $G_{\varphi}' = F_n.$ Similarly
we can to prove that $\gamma_i G_{\varphi} = F_n$ for all $i > 1.$
\end{proof}

\subsection{Proof of Theorem \ref{resnilp}.}
 Let $F_n$ be a free group of rank $n\geq 2$ with a
free generator set $\{x_1, x_2,...,x_n\}$ and $Aut(F_n)$ the group
of automorphisms of $F_n$. Consider the subgroup of $Aut(F_n)$,
generated by automorphisms of the form
$$
\varepsilon_{ij} : \left\{
\begin{array}{ll}
x_{i} \longmapsto x_{j}^{-1}x_ix_j & \mbox{if }~~ i\neq j, \\
x_{l} \longmapsto x_{l} & \mbox{if }~~ l\neq i.
\end{array} \right.
$$
and  denote this subgroup by $Cb_n.$ This is the group {\it basis
conjugating automorphisms} of a free group. It is
torsion-free nilpotent \cite{A} and also
has topological interpretations. It is the
pure group of motions of $n$ unlinked circles in $S^3$ \cite{DG,
JMM}. On
the other hand it is also the pure braid-permutation group
\cite{CPVW}.
 This group is denoted
$P\varSigma_n$ in \cite{JMM} and in \cite{CPVW}. McCool gave the
following presentation for it \cite{MC}:
\begin{multline}
\langle \varepsilon_{ij},\ 1\leq i\neq j \leq n\ |\
\varepsilon_{ij}\varepsilon_{kl}=\varepsilon_{kl}\varepsilon_{ij},\\
\varepsilon_{ij}\varepsilon_{kj}=\varepsilon_{kj}\varepsilon_{ij},
\varepsilon_{ij}\varepsilon_{kj}\varepsilon_{ik}=\varepsilon_{ik}\varepsilon_{ij}
\varepsilon_{kj} \rangle
\end{multline}
where distinct letters stand for distinct indexes.

There exist a homomorphism
$$
\varphi : PV_3 \longrightarrow Cb_3, ~~~\varphi(\lambda_{ij}) =
\varepsilon_{ij},~~~1 \leq i \not= j \leq 3.
$$
We will denote the images of elements $a_i,$ $b_i,$ $c_i,$ by
$\alpha_i,$ $\beta_i,$ $\gamma_i,$ $i=1,2,$ respectively.

In $Cb_3$ hold the set of defining relations from $PV_3$ and
relations
$$
\varepsilon_{13} \varepsilon_{23} = \varepsilon_{23}
\varepsilon_{13},~~~ \varepsilon_{12} \varepsilon_{32} =
\varepsilon_{32} \varepsilon_{12},~~~ \varepsilon_{21}
\varepsilon_{31} = \varepsilon_{31} \varepsilon_{21}.
$$
Since
\begin{align*}
& \varepsilon_{13} = \gamma_2,~~~\varepsilon_{31} = \gamma_2^{-1}
\gamma_1,~~~\varepsilon_{12} = \gamma_2^{-1} \beta_1,\\
& \varepsilon_{21} = \beta_2 \gamma_1^{-1} \gamma_2,\
\varepsilon_{23} = \gamma_2^{-1} \alpha_1,\ \varepsilon_{32} =
\alpha_2 \gamma_1^{-1} \gamma_2,
\end{align*}
this set of defining relations has the form
$$
\alpha_1 = \gamma_2^{-1} \alpha_1 \gamma_2,~~~\gamma_2^{-1}
\beta_1 \alpha_2 \gamma_1^{-1} \gamma_2 = \alpha_2 \gamma_1^{-1}
\beta_1,~~~ \beta_2 = \gamma_2^{-1} \gamma_1 \beta_2 \gamma_1^{-1}
\gamma_2.
$$
Rewrite this set as follows:
$$
\alpha_1^{\gamma_2} = \alpha_1,~~~(\beta_1 \alpha_2
\gamma_1^{-1})^{\gamma_2} = \alpha_2 \gamma_1^{-1} \beta_1,~~~
(\gamma_1 \beta_2 \gamma_1^{-1})^{\gamma_2} = \beta_2.
$$
We have
\begin{lem}
$$
Cb_3 = \langle G_3, \gamma_2 \ |\  \gamma_2^{-1} \, A \, \gamma_2
= B, ~~\psi \rangle
$$
is an HNN-extension with associated subgroups
$$
A = \langle \alpha_1, \beta_1 \alpha_2 \gamma_1^{-1}, \gamma_1
\beta_2 \gamma_1^{-1} \rangle,~~~ B = \langle \alpha_1, \alpha_2
\gamma_1^{-1} \beta_1, \beta_2  \rangle,
$$
and the isomorphism $\psi : A \longrightarrow B$ is defined by the
rule
$$
\psi : \left\{
\begin{array}{lcl}
\alpha_1 & \longrightarrow  & \alpha_1, \\
\beta_1 \alpha_2 \gamma_1^{-1} & \longrightarrow  & \alpha_2 \gamma_1^{-1} \beta_1, \\
\gamma_1 \beta_2 \gamma_1^{-1} & \longrightarrow  & \beta_2. \\
\end{array}
\right.
$$
Subgroup $G_3$ is isomorphic to subgroup from $PV_3.$
\end{lem}
\begin{proof}
We have to prove that $A \simeq B.$ Find  the sets of generators
of $B$ and $A$ in old generators of $Cb_3.$ We have
$$
B = \langle \varepsilon_{13} \varepsilon_{23},~~\varepsilon_{32}
\varepsilon_{12},~~ \varepsilon_{21} \varepsilon_{31} \rangle
$$
and we see that $\varepsilon_{13} \varepsilon_{23}$ is an
automorphism of $F_3$ which is the conjugation by $x_3,$
$\varepsilon_{32} \varepsilon_{12}$ is an automorphism of $F_3$
which is the conjugation by $x_2,$ $\varepsilon_{21}
\varepsilon_{31}$ is an automorphism of $F_3$ which is the
conjugation by $x_1.$ Hence
$$
B = \langle \widehat{x_1}, \, \widehat{x_2}, \,\widehat{x_3}
\rangle \simeq F_3,
$$
where $\widehat{y}$ is an inner automorphism of $F_3$ which is
conjugation by $y$.

Similar,
$$
A = \langle \widehat{x_2}, \, \widehat{x_3}, \,\widehat{x_3 x_1
x_3^{-1}} \rangle \simeq F_3,
$$
and $A \simeq B.$
\end{proof}

The group $G_3$ is residually torsion free
nilpotent as a subgroup of residually torsion free nilpotent
group  $Cb_3$.  Recall the following result of Malcev from
\cite{M}: the free product of residually torsion-free nilpotent
groups is residually torsion-free nilpotent. Hence,
Proposition~\ref{decomp} and the fact that $G_3$ is residually
torsion free nilpotent imply the statement of Theorem
\ref{resnilp}.\ $\Box$

\subsection{Remark on finite type invariants} Recall the notion of
finite type invariants for virtual braids (see \cite{BB}). Let $A$
be an abelian group and $n\geq 2$. Consider a set map $v: VB_n\to
A$, i.e. an $A$-valued invariant of virtual $n$-braids. Let $J$ be
the two-sided ideal in the integral group ring $\mathbb Z[VB_n]$
generated by elements
$\{\sigma_i-\rho_i, \ \sigma_i^{-1}-\rho_i\ |\ i=1,\dots, n-1\}$.
The filtration of the group ring
$$
\mathbb Z[VB_n]\supset J\supset J^2\supset \dots
$$
is called {\it Goussarov-Polyak-Viro filtration}. We say that an
invariant $v: VB_n\to A$ is of degree $d$ if its linear extension
$\mathbb Z[VB_n]\to A$ vanishes on $J^{d+1}$.
As usual the augmentation ideal is defined by the formula
$\Delta(PV_n)=\mathrm{Ker}\{\mathbb
Z[PV_n]\to \mathbb Z\}$.
The
Goussarov-Polyak-Viro filtration for $VB_n$ corresponds to the
filtration by powers of augmentation ideal
of the group ring $\mathbb Z[PV_n]:$
$$
\mathbb Z[PV_n]\supset \Delta\supset \Delta^2\supset\dots.
$$
Theorem~\ref{resnilp} implies that  the intersection of
augmentation powers of the group ring $\mathbb Z[PV_3]$ is zero.
Hence, we have the following:
\begin{prop}
The set of invariants of finite degree is complete for virtual
pure braids on three strands. \ $\Box$
\end{prop}

\section{Relations among relations}

\subsection{Presentation of $PV_3$.}
\begin{thm}
The presentation (\ref{asphpres}) is aspherical.
\end{thm}
\begin{proof}
Consider the standard 2-complex $K$ constructed for the
presentation (\ref{asphpres}). We have the following 
exact sequence of abelian groups, called also Hopf exact sequence:
\begin{equation}
0\to H_3(PV_3)\to \pi_2(K)_{PV_3}\to H_2(K)\to H_2(PV_3)\to 0
\end{equation}
One can 
obtain this sequence by applying the homology
functor $H_*(PV_3,-)$ to the chain complex $C_*(\tilde K)$ of the
universal covering space $\tilde K$ of $K$.

It is easy to prove that $H_2(PV_3)=\mathbb Z^{\oplus 6}$. For
that we can either use the result from \cite{BEER}, where all
homology of virtual pure braid groups are described, or consider
the semidirect product decomposition (\ref{semi}) and observe that
$H_2(PV_3)=H_1(V_1, H_1(V_2^{*}))$. Since $PV_3$ is a semidirect
product of two free groups, its third homology is zero. Now
observe that $H_2(K)$ is a free abelian group of a rank less or
equal to six, since there are only six 2-dimensional cells in the $2$-complex $K$.
Hence, the map $H_2(K)\to H_2(PV_3)$ is an isomorphism and
therefore, $\pi_2(K)_{PV_3}=0$. Let $\alpha\in \pi_2(K)$. Since
$\pi_2(K)_{PV_3}=0,$ for every $n\geq 1$, the element $\alpha$ can
be presented as a finite sum
\begin{equation}\label{d}
\alpha=\sum_i (1-g_1^{i})\dots (1-g_n^{i})\alpha_i
\end{equation}
for some elements $g_j^{i}\in PV_3$. Here $(1-g)\beta=\beta-g\circ
\beta,\ g\in PV_3,\ \beta\in \pi_2(K)$ and the action of
$PV_3=\pi_1(K)$ on $\pi_2(K)$ is standard. Consider the
monomorphism of $\mathbb Z[PV_3]$-modules:
$$
f: \pi_2(K)=H_2(\tilde K)\hookrightarrow C_2(\tilde K)\cong
\mathbb Z[PV_3]^{\oplus 6},
$$
where $C_2(\tilde K)$ is the second term of the chain complex
$C_*(\tilde K)$. The coordinates of the monomorphism $f$ are
elements of $\mathbb Z[PV_3]$ which lie in the intersection of
powers of augmentation ideal $\Delta(PV_3)
$
due to the existence of the presentation
(\ref{d}) for every $n\geq 1$. However, for a residually
torsion-free nilpotent group, the intersection of powers of
augmentation ideals is zero \cite{H}, hence $\alpha=0$ by
Theorem~\ref{resnilp}. Hence, the presentation (\ref{asphpres}) is
aspherical.
\end{proof}

\subsection{Presentation of $PV_n,\ n\geq 4$.} For $n\geq 1,$ as we mentioned in introduction, the
groups $PV_n$ admit the following presentation
\begin{multline}\label{higher} \langle \lambda_{ij},\ 1\leq i\neq j\leq n\
|\ [\lambda_{ij},\lambda_{kl}]=1,\
\lambda_{ki}\lambda_{kj}\lambda_{ij}\lambda_{ki}^{-1}\lambda_{kj}^{-1}\lambda_{ij}^{-1}=1\rangle
\end{multline}
where distinct letters stand for distinct indices. The classifying
spaces of groups $PV_n$ are constructed in \cite{BEER} as
natural quotients of unions of permutohedra.

\begin{figure}
\begin{center}
\epsfxsize 120mm \epsfbox{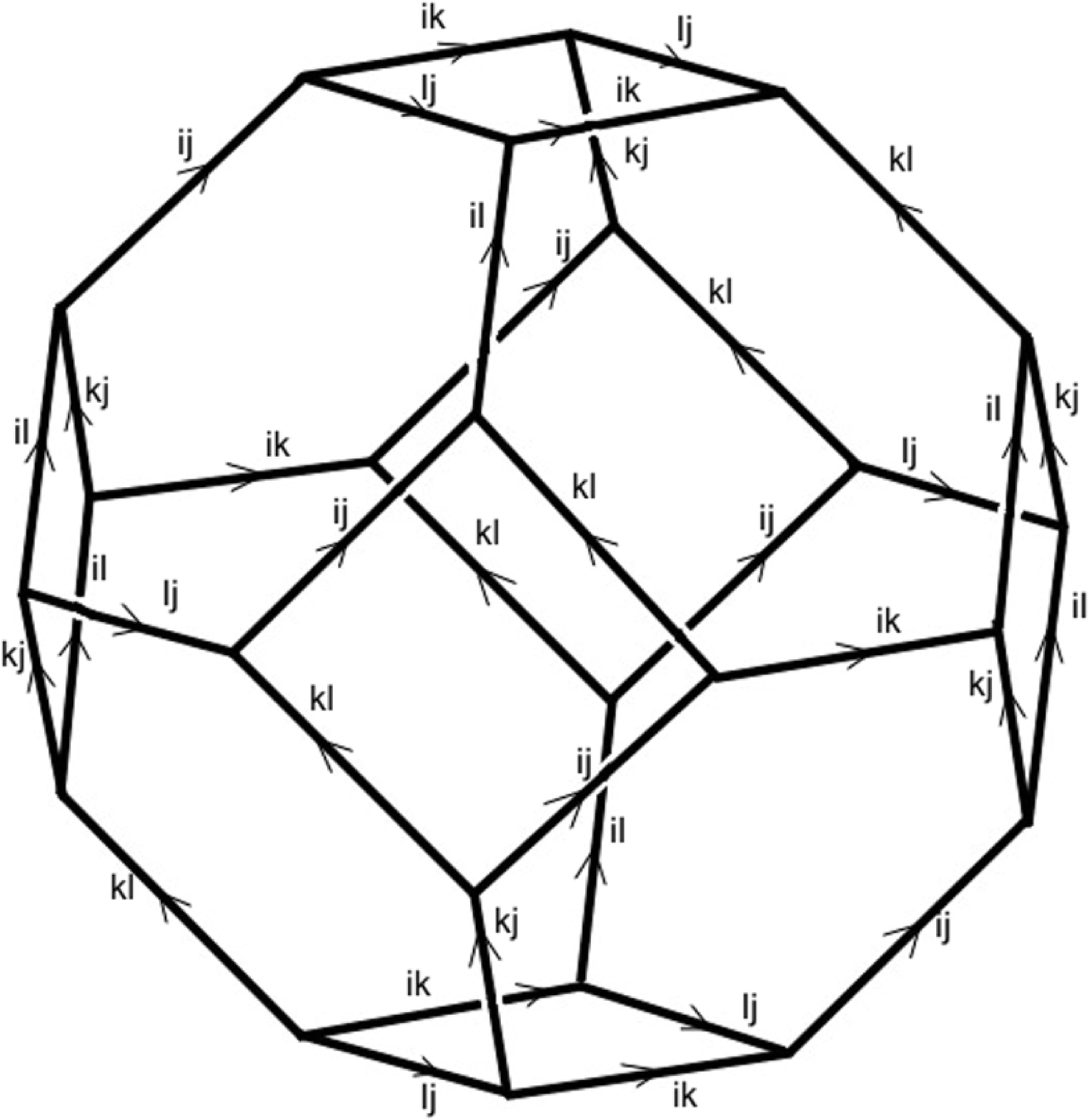}
\end{center}
\caption{Truncated octahedron and relation among relations from
(\ref{higher})} \label{f1}
\end{figure}

Recall that, by a result of Deligne \cite{D} and Salvetti
\cite{Sal}, the classifying space of a braid group on $n$ strands
can be constructed as a quotient of the $n$-th permutohedron. In
\cite{Lo}, the interpretation for this construction is given in
terms of homotopical syzygies for presentations of braid groups.
It is shown in \cite{Lo} that  the relations among relations in
usual presentation of braid groups can be viewed as labellings of
the permutohedron $P_4,$ or truncated octahedron. Here we do the
same for the groups $PV_n,\ n\geq 4.$ The relations from
(\ref{higher}) are of length four and six. The  labelling of the
truncated octahedron presented in the Figure 1 gives the needed
relation. In the Figure 1, the labeling $ij$ of an arrow means
that the generator $\lambda_{ij}$ corresponds to this arrow.

Now one can follow the ideology of the paper \cite{Lo} and find
the interpretation of the construction of the classifying space
for $PV_n$ from \cite{BEER} in the same way as the result of
Deligne \cite{D} and Salvetti \cite{Sal} is interpreted in
\cite{Lo}. It is clear that the higher permutohedra appear in the
construction of higher homotopical syzygies for the presentations
(\ref{higher}).

\section{Cohomology Ring of $PV_3$}

\subsection{A Representation of the Ring $H^*(G_3)$}
Let $X_1=T\{a_1,b_1\}$ be the torus $T$ labeled by $\{a_1,b_1\}$. Let $\{x_1,x_2\}$ be the standard basis for $\pi_1(X_1)=\pi_1(T)$. According to the defining relation $[a_1,b_1]=1$ in $G_3$, there is a map
$$
\phi_1\colon X_1\longrightarrow BG_3
$$
such that in the fundamental groups
\begin{equation}\label{equation5.1}
\phi_{1\ast}(x_1)=a_1\textrm{ and } \phi_{1\ast}(x_2)=b_1.
\end{equation}
Similarly, from the defining relation $[a_2,b_2]=1$ in $G_3$, there is a torus $X_2=T\{a_2,b_2\}$ with a map
$$
\phi_2\colon X_2\longrightarrow BG_3
$$
such that in the fundamental groups
\begin{equation*}
\phi_{2\ast}(x_3)=a_2\textrm{ and } \phi_{2\ast}(x_4)=b_2,
\end{equation*}
where $\{x_3,x_4\}$ is the standard basis for $\pi_1(X_2)$.
Consider the defining relation
$$
\begin{array}{cl}
 &b_1^{c_1}=b_1^{a_2}\\
\Longleftrightarrow&  [b_1,c_1]=[b_1,a_2]\\
\Longleftrightarrow& [b_1,c_1][a_2,b_1]=1\\
\end{array}
$$
in $G_3$. Let $X_3=S_2\{b_1,c_1;a_2,b_1\}$ be the closed oriented
surface of genus $2$ labeled by $\{b_1,c_1;a_2,b_1\}$ with
standard generators $\{y_{11}, z_{11},y_{12},z_{12}\}$ of its
fundamental group $\pi_1(X_3)$ and defining relation
$[y_{11},z_{11}][y_{12},z_{12}]=1$. Then there is map
$$
\phi_3\colon X_3\longrightarrow BG_3
$$
such that in the fundamental groups
\begin{equation*}
\phi_{3\ast}(y_{11})=b_1,\ \phi_{3\ast}(z_{11})=c_1,\ \phi_{3\ast}(y_{12})=a_2 \textrm{ and } \phi_{3\ast}(z_{12})=b_1.
\end{equation*}
Similar constructions apply to the remaining three defining relations in equation~(\ref{eq:relcong}). So we have the spaces
\begin{equation}\label{equation5.3}
\begin{array}{c}
X_4=S_2\{a_1,c_1;b_2,a_1\},\\
X_5=S_2\{b_2,c_1;a_1b_2,b_2\}\textrm{ and }\\
X_6=S_2\{a_2,c_1;b_1a_2,a_2\}\\
\end{array}
\end{equation}
with the maps $\phi_i\colon X_i \to BG_3$, $i=4,5,6$, such that in fundamental groups $\phi_{i\ast}$ sends the standard basis $\{y_{(i-2)1},z_{(i-2)1},y_{(i-2)2},z_{(i-2)2}\}$ into the corresponding words in the labeling bracket for $X_i$, where the standard defining relation in $\pi_1(X_i)$ is $[y_{(i-2)1},z_{(i-2)1}][y_{(i-2)2},z_{(i-2)2}]=1$. For instance, $\phi_{5\ast}(y_{31})=a_1b_2$.

Let $X=\bigvee\limits_{i=1}^6 X_i$ and let
\begin{equation}\label{equation5.4}
\phi\colon X\longrightarrow BG_3
\end{equation}
be the map such that $\phi|_{X_i}=\phi_i$ for $1\leq i\leq 6$. The abelianization $G_3\to G_3^{\mathrm{ab}}=\Z^{\oplus 5}$ induces a map
\begin{equation}\label{equation5.5}
\psi\colon BG_3\longrightarrow B\Z^{\oplus 5}=T^5.
\end{equation}

\begin{lem}\label{lemma5.1}
The following properties hold:
\begin{enumerate}
\item[1)] The algebra map
$$
\phi^*\colon H^*(BG_3)\longrightarrow H^*(X)
$$
is a monomorphism. Moreover
$$
\phi^*\colon H^2(BG_3)\longrightarrow H^2(X)
$$
is an isomorphism of abelian groups.
\item[2)] The algebra map
$$
\psi^*\colon H^*(T^5)\longrightarrow H^*(BG_3)
$$
is an epimorphism.
\end{enumerate}
\end{lem}
\begin{proof}
Let $\theta=\psi\circ\phi\colon X\to T^5$. We first prove that
$$
\theta^*\colon H^2(T^5)\longrightarrow H^2(X)
$$
is onto. Observe that $H_1(X)$ is the free abelian group with a basis
$$
\{x_i, y_{i1}, z_{i1},y_{i2},z_{i2}\ | \ 1\leq i\leq 4\}
$$
and $H_1(T^5)=H_1(BG_3)$ is the free abelian group with a basis $$\{a_1,b_1,a_2,b_2,c_1\},$$ where we use the same notation for the induced generators in the first homology group from the fundamental group. From the construction of the map $\phi$, the linear transformation
$$
\theta_*\colon H_1(X)\to H_1(T^5)
$$
is given by the following formula:
\begin{equation}\label{equation5.7}
\begin{array}{rclrclrclrcl}
\theta_*(x_1)&=& a_1, &\theta_*(x_2)&=& b_1, &\theta_*(x_3)&=& a_2, &\theta_*(x_4)&=& b_2,\\
\theta_*(y_{11})&=& b_1, &\theta_*(z_{11})&=& c_1, &\theta_*(y_{12})&=& a_2, &\theta_*(z_{12})&=& b_1,\\
\theta_*(y_{21})&=& a_1, &\theta_*(z_{21})&=& c_1, & \theta_*(y_{22})&=& b_2, &\theta_*(z_{22})&=& a_1,\\
\theta_*(y_{31})&=& b_2, &\theta_*(z_{31})&=& c_1, &\theta_*(y_{32})&=& a_1+b_2, &\theta_*(z_{32})&=& b_2,\\
\theta_*(y_{41})&=& a_2, &\theta_*(z_{41})&=& c_1, &\theta_*(y_{42})&=& b_1+a_2, &\theta_*(z_{42})&=& a_2.\\
\end{array}
\end{equation}
Consider the dual basis
$$
\{x_i^*, y_{i1}^*, z_{i1}^*,y_{i2}^*,z_{i2}^*\ | \ 1\leq i\leq 4\}
$$
for $H^1(X)$ and
$$\{a_1^*,b_1^*,a_2^*,b_2^*,c_1^*\}$$
for $H^1(T^5)$. The linear transformation $\theta^*\colon H^1(T^5)\to H^1(X)$ is given by the formula
\begin{equation}\label{equation5.8}
\begin{array}{rcl}
\theta^*(a_1^*)&=& x_1^*+y_{21}^*+z_{22}^*+y_{32}^*,\\
\theta^*(b_1^*)&=&x_2^*+y_{11}^*+z_{12}^*+y_{42}^*,\\
\theta^*(a_2^*)&=&x_3^*+y_{12}^*+y_{41}^*+y_{42}^*+z_{42}^*,\\
\theta^*(b_2^*)&=&x_4^*+y_{22}^*+y_{31}^*+y_{32}^*+z_{32}^*,\\
\theta^*(c_1)&=&z_{11}^*+z_{21}^*+z_{31}^*+z_{41}^*.\\
\end{array}
\end{equation}
From the algebra structure on $H^*(X)$, we have
\begin{equation}\label{equation5.9}
\begin{array}{rclrclrcl}
\theta^*(a_1^*b_1^*)&=&x_1^*x_2^*,&\theta^*(a_2^*b_2^*)&=&x_3^*x_4^*,&\theta^*(a_1^*c_1^*)&=&y_{21}^*z_{21}^*,\\
\theta^*(b_1^*c_1^*)&=&y_{11}^*z_{11}^*,&\theta^*(a_2^*c_1^*)&=&y_{41}^*z_{41}^*,&\theta^*(b_2^*c_1^*)&=&y_{31}^*z_{31}^*.\\
\end{array}
\end{equation}
Since
$$
\left\{x_1^*x_2^*,x_3^*x_4^*,y_{21}^*z_{21}^*,y_{11}^*z_{11}^*,y_{41}^*z_{41}^*,y_{31}^*z_{31}^*\right\}
$$
forms a basis for $H^2(X)$, the linear map
$$
\theta^*\colon H^2(T^5)\longrightarrow H^2(X)
$$
is onto.

Now since $\theta^*=\phi^*\circ\psi^*$, the linear map
$$
\phi^*\colon H^2(BG_3)\longrightarrow H^2(X)
$$
is onto and so it is an isomorphism because $$H^2(BG_3)\cong H^2(PV_3)\cong H_2(PV_3)\cong \Z^{\oplus 6}$$ and $H^2(X)\cong \Z^{\oplus 6}$. This proves assertion 1 because clearly $$\phi^*\colon H^1(BG_3)\to H^1(X)$$ is a monomorphism. Assertion 2 also follows because $H^i(BG_3)=0$ for $i>2$.
\end{proof}

\subsection{The Cohomology Ring $H^*(G_3)$}
From formulas~(\ref{equation5.8}) and~(\ref{equation5.9}), we have
$$
\begin{array}{rcl}
\theta^*(b_2^*a_1^*)&=& (x_4^*+y_{22}^*+y_{31}^*+y_{32}^*+z_{32}^*)(x_1^*+y_{21}^*+z_{22}^*+y_{32}^*)\\
&=&y_{22}^*z_{22}^*+z_{32}^*y_{32}^*\\
&=&y_{21}^*z_{21}^*-y_{31}^*z_{31}^*\\
&=&\theta^*(a_1^*c_1^*)-\theta^*(b_2^*c_1^*)\\
\end{array}
$$
and so we have the relation
\begin{equation}\label{equation5.10}
a_1^*c_1^*+a_1^*b_2^*+c_1^*b_2^*\equiv 0
\end{equation}
in $H^*(BG_3)$. Similarly, by computing $\theta^*(a_2^*b_1^*)$, we have
\begin{equation}\label{equation5.11}
b_1^*c_1^*+b_1^*a_2^*+c_1^*a_2^*\equiv 0
\end{equation}
in $H^*(BG_3)$. From formula~(\ref{equation5.8}), we have
$$
\begin{array}{rcl}
\theta^*(b_2^*b_1^*)&=&(x_4^*+y_{22}^*+y_{31}^*+y_{32}^*+z_{32}^*)(x_2^*+y_{11}^*+z_{12}^*+y_{42}^*)\\
&=&0.\\
\end{array}
$$
Similarly, $\theta^*(a_2^*a_1^*)=0$ and so we have the relations
\begin{equation}\label{equation5.12}
a_1^*a_2^*=b_1^*b_2^*\equiv0
\end{equation}
in  $H^*(BG_3)$.

\begin{thm}\label{theorem5.2}
The cohomology ring $H^*(G_3)$ is the quotient algebra of the exterior algebra
$$
E(a_1^*,b_1^*,a_2^*,b_2^*,c_1^*)
$$
subject to the four defining relations given in formulas~(\ref{equation5.10}), ~(\ref{equation5.11}) and ~(\ref{equation5.12}).
\end{thm}
\begin{proof}
Let $A$ be the quotient algebra of $H^*(T^5)=E(a_1^*,b_1^*,a_2^*,b_2^*,c_1^*)$ subject to the four defining relations given in formulas~(\ref{equation5.10}), ~(\ref{equation5.11}) and ~(\ref{equation5.12}). Then
$$
A^i\cong H^i(BG_3)
$$
for $i\leq 2$. It suffices to check that $A^3=0$. From formula~(\ref{equation5.12}), we have
$$
a_1^*a_2^*b_1^*\equiv a_1^*a_2^*b_2^*\equiv a_1^*a_2^*c_1^*\equiv b_1^*b_2^*a_1^*\equiv b_1^*b_2^*a_2^*\equiv b_1^*b_2^*c_1^*\equiv 0.
$$
From formulas~(\ref{equation5.10}) and~(\ref{equation5.12}), we have
$$
a_1^*c_1^*b_1^*\equiv a_1^*c_1^*b_2^*\equiv 0.
$$
From formulas~(\ref{equation5.11}) and~(\ref{equation5.12}), we have $a_2^*c_1^*b_1^*\equiv 0$ and
$$
a_2^*c_1^*b_2^*\equiv (b_1^*c_1^*-b_1^*a_2^*)b_2^*\equiv 0.
$$
Thus $A^3=0$ and hence the result.
\end{proof}

\subsection{The Cohomology Ring $H^*(PV_3)$}
From the decomposition of $PV_3$ in Proposition~\ref{decomp}, there is a group isomorphism
$$
\delta\colon G_3\ast \la c_2\ra=G_3\ast\Z\longrightarrow PV_3
$$
with
$$
\begin{array}{cc}
\delta(a_1)=\lambda_{13}\lambda_{23},& \delta(b_1)=\lambda_{13}\lambda_{12},\\
\delta(a_2)=\lambda_{32}\lambda_{31},& \delta(b_2)=\lambda_{21}\lambda_{31},\\
\delta(c_1)=\lambda_{13}\lambda_{31},&\delta(c_2)=\lambda_{13}.\\
\end{array}
$$
Thus
$$
\delta_*\colon H_1(G_3\ast\Z)\longrightarrow H_1(PV_3)
$$
is the linear transformation given by
\begin{equation}\label{equation5.13}
\begin{array}{cc}
\delta_*(a_1)=\lambda_{13}+\lambda_{23},& \delta_*(b_1)=\lambda_{13}+\lambda_{12},\\
\delta_*(a_2)=\lambda_{32}+\lambda_{31},& \delta_*(b_2)=\lambda_{21}+\lambda_{31},\\
\delta_*(c_1)=\lambda_{13}+\lambda_{31},&\delta_*(c_2)=\lambda_{13}\\
\end{array}
\end{equation}
and so its dual $\delta^*\colon H^1(PV_3)\to H^1(G_3\ast\Z)$ is given by the formula
\begin{equation}\label{equation5.14}
\begin{array}{rcl}
\delta^*(\lambda_{13}^*)&=&a_1^*+b_1^*+c_1^*+c_2^*,\\
\delta^*(\lambda_{31}^*)&=&a_2^*+b_2^*+c_1^*,\\
\delta^*(\lambda_{12}^*)&=&b_1^*,\\
\delta^*(\lambda_{21}^*)&=&b_2^*,\\
\delta^*(\lambda_{23}^*)&=&a_1^*,\\
\delta^*(\lambda_{32}^*)&=&a_2^*.\\
\end{array}
\end{equation}
It follows that $(\delta^*)^{-1}\colon H^1(G_3\ast\Z)\to H^1(PV_3)$ is given by
\begin{equation}\label{equation5.15}
\begin{array}{rcl}
(\delta^*)^{-1}(a_1^*)&=&\lambda_{23}^*,\\
(\delta^*)^{-1}(b_1^*)&=&\lambda_{12}^*,\\
(\delta^*)^{-1}(a_2^*)&=&\lambda_{32}^*,\\
(\delta^*)^{-1}(b_2^*)&=&\lambda_{21}^*,\\
(\delta^*)^{-1}(c_1^*)&=&\lambda_{31}^*-\lambda_{32}^*-\lambda_{21}^*,\\
(\delta^*)^{-1}(c_2^*)&=&(\lambda_{13}^*-\lambda_{31}^*)+(\lambda_{32}^*-\lambda_{23}^*)+(\lambda_{21}^*-\lambda_{12}^*).\\
\end{array}
\end{equation}
In the cohomology ring $H^*(G_3\ast\Z)=H^*(BG_3\vee S^1)$, we have
$$
c_2^*\alpha=0
$$
for any $\alpha\in H^1(G_3\ast\Z)$. Thus we have relations
\begin{equation}\label{equation5.16}
(\lambda_{13}^*-\lambda_{31}^*)\lambda_{ij}^*\equiv (\lambda_{12}^*-\lambda_{21}^*)\lambda_{ij}^*
+(\lambda_{23}^*-\lambda_{32}^*)\lambda_{ij}^*
\end{equation}
in $H^*(PV_3)$ for $1\leq i\not=j\leq 3$. From formula~(\ref{equation5.12}), we have
\begin{equation}\label{equation5.17}
\lambda_{23}^*\lambda_{32}^*\equiv\lambda_{12}^*\lambda_{21}^*\equiv0
\end{equation}
in $H^*(PV_3)$. Together with formulas~(\ref{equation5.10}) and~(\ref{equation5.11}), we have
\begin{equation}\label{equation5.18}
\begin{array}{c}
\lambda_{21}^*\lambda_{31}^*\equiv \lambda_{21}^*\lambda_{32}^*+ \lambda_{23}^*\lambda_{31}^*,\\
\lambda_{21}^*\lambda_{32}^*\equiv \lambda_{12}^*\lambda_{31}^*+ \lambda_{31}^*\lambda_{32}^*\\
\end{array}
\end{equation}
in $H^*(PV_3)$.
\begin{thm}\label{theorem5.3}
The cohomology ring $H^*(PV_3)$ is the quotient algebra of the exterior algebra
$$
E(\lambda_{12}^*,\lambda_{21}^*,\lambda_{13}^*,\lambda_{31}^*,\lambda_{23}^*,\lambda_{32}^*)
$$
subject to the following relations
\begin{enumerate}
\item[1)] $\lambda_{ij}^*\lambda_{ji}^*\equiv 0$ for $1\leq i<j\leq 3$,
\item[2)] $(\lambda_{13}^*-\lambda_{31}^*)\lambda_{ij}^*\equiv (\lambda_{12}^*-\lambda_{21}^*)\lambda_{ij}^*+(\lambda_{23}^*-\lambda_{32}^*)\lambda_{ij}^*$ for $1\leq i\not=j\leq 3$ and
\item[3)] $\lambda_{21}^*\lambda_{31}^*\equiv \lambda_{21}^*\lambda_{32}^*+ \lambda_{23}^*\lambda_{31}^*$.
\end{enumerate}
\end{thm}
\begin{proof}
According to the above computations, the defining relations for $H^*(PV_3)$ are given by the formulas~(\ref{equation5.16}), ~(\ref{equation5.17}) and~(\ref{equation5.18}). The assertion follows from the fact by adding two equations in formula~(\ref{equation5.18}) together with using formula~(\ref{equation5.16}), formula~(\ref{equation5.18}) can be placed by
$$
\left\{
\begin{array}{c}
\lambda_{21}^*\lambda_{31}^*\equiv \lambda_{21}^*\lambda_{32}^*+ \lambda_{23}^*\lambda_{31}^*,\\
\lambda_{13}^*\lambda_{31}^*\equiv0.\\
\end{array}
\right.
$$
\end{proof}

\noindent\textbf{Note.} In the cohomology ring $H^*(G_3\ast\Z)$, the element $c_2^*$ creates five relations
$$
a_1^*c_2^*\equiv a_2^*c_2^*\equiv b_1^*c_2^*\equiv b_2^*c_2^*\equiv c_1^*c_2^*\equiv 0
$$
among the defining relations for $H^*(G_3\ast\Z)$. Thus the six equations in condition 2 are linearly dependent.

\section{The Associated Graded Lie Algebra for $PV_3$}
Let $G$ be a group and let
$$
L(G)=\bigoplus_{i=1}^\infty\gamma_i(G)/\gamma_{i+1}(G)
$$
be the associated graded Lie algebra of $G$. Consider the short
exact sequence
$$
\langle a_1,b_1,a_2,b_2\rangle \rightarrowtail G_3
\twoheadrightarrow \mathbb{Z}=\langle c_1\rangle.
$$
Since $c_1$ acts trivially on $\langle
a_1,b_1,a_2,b_2\rangle^{\mathrm{ab}}$, there is a short exact
sequence of Lie algebras
$$
L(\langle a_1,b_1,a_2,b_2\rangle)\rightarrowtail L(G_3)
\twoheadrightarrow L(\langle c_1\rangle)
$$
by~\cite{FR1}. Let $A_1,A_2,B_1,B_2,C_1$ be the elements in the Lie
algebra induced by $a_1,b_1,a_2,b_2,c_1$, respectively. Observe that
$L(\langle a_1,b_1,a_2,b_2\rangle)=L(\mathbb{Z}^2\ast \mathbb{Z}^2)$
is the quotient of the free Lie algebra
$$
L(A_1,B_1,A_2,B_2)
$$
by the relations
\begin{equation}\label{Lie-relation1}
\begin{array}{cccc}
&[A_1,B_1]&=&0\\
&[A_2,B_2]&=&0.\\
\end{array}
\end{equation}
From the defining relations of $G_3$, we have the following
identities:
\begin{equation}\label{Lie-relation2}
\begin{array}{cccc}
&[C_1,B_1]&=&[A_2,B_1]\\
&[C_1,A_1]&=&[B_2,A_1]\\
&[C_1,B_2]&=&[A_1,B_2]\\
&[C_1,A_2]&=&[B_1,A_2]\\
\end{array}
\end{equation}

\begin{thm}
The Lie algebra $L(PV_3)$ is the quotient of the free Lie algebra
$L(A_1,B_1,A_2,B_2,C_1, C_2)$ subject to the relations defined by
(~\ref{Lie-relation1}) and (~\ref{Lie-relation2}).
\end{thm}
\begin{proof}
From the decomposition $PV_3\cong G_3\ast \langle c_2\rangle$, it
suffices to show that $L(G_3)$ is the quotient of
$L(A_1,B_1,A_2,B_2,C_1)$ subject to the relations defined by
(~\ref{Lie-relation1}) and (~\ref{Lie-relation2}).

Let $L$ be the quotient of the free Lie algebra
$L(A_1,B_1,A_2,B_2,C_1)$ subject to the relations defined by
(~\ref{Lie-relation1}) and (~\ref{Lie-relation2}). Then the morphism
of Lie algebras
$$
L(A_1,B_1,A_2,B_2,C_1)\twoheadrightarrow L(G_3)
$$
factors through $L$. Thus is an epimorphism of Lie algebras
$$
\phi\colon L\twoheadrightarrow L(G_3).
$$
Let $L'$ be the sub Lie algebra of $L$ generated by
$A_1,B_1,A_2,B_2$. Then $L'$ is a Lie ideal of $L$ by the defining
relations of $L$ with a short exact sequence of Lie algebras
$$
L'\rightarrowtail L\twoheadrightarrow L(C_1).
$$
The assertion follows by applying 5-Lemma to the following
commutative diagram
$$
\begin{array}{ccccc}
L'& \rightarrowtail &L& \twoheadrightarrow &L(C_1)\\
\downarrow & &\downarrow && \|\\
L(\langle a_1,b_1,a_2,b_2\rangle)&\rightarrowtail& L(G_3)&
\twoheadrightarrow& L(\langle c_1\rangle).\\
\end{array}
$$
\end{proof}

\end{document}